\documentclass[a4paper,10pt]{article}

\usepackage[margin=20mm,nohead]{geometry}
\usepackage{latexsym,amssymb,graphics,graphicx,epsfig,color,enumerate}
\usepackage{amsmath,amsthm}

\usepackage{t1enc}
\usepackage[latin2]{inputenc}

\theoremstyle{plain}
\newtheorem{thm}{Theorem}

\newtheorem{cor}[thm]{Corollary}
\newtheorem{cl}[thm]{Claim}

\newtheorem{conj}[thm]{Conjecture}

\theoremstyle{definition}

\title{A note on $\mathtt{V}$-free $2$-matchings}
\author{
Krist\'of B\'erczi\thanks{MTA-ELTE Egerv\'ary Research Group, Department of Operations Research, E\"otv\"os University, Budapest, Hungary. E-mail: {\tt berkri@cs.elte.hu}.}
\and
Attila Bern\'ath\thanks{MTA-ELTE Egerv\'ary Research Group, Department of Operations Research, E\"otv\"os University, Budapest, Hungary. E-mail: {\tt bernath@cs.elte.hu}.}
\and
M\'at\'e Vizer\thanks{Alfr\'ed R\'enyi Institute of Mathematics, P.O.B. 127, Budapest H-1364, Hungary. Email:
{\tt vizermate@gmail.com}.}}

\newcommand{\cE}{\ensuremath{\mathcal{E}}}
\newcommand{\cF}{\ensuremath{\mathcal{F}}}

\begin{document}
\maketitle

\begin{abstract}
Motivated by a conjecture of Liang [Y.-C. Liang. {\em Anti-magic labeling of graphs}. PhD thesis, National Sun Yat-sen University, 2013.], we introduce a restricted path packing problem in bipartite graphs that we call a $\mathtt{V}$-free $2$-matching. We verify the conjecture through a weakening of the hypergraph matching problem. We close the paper by showing that it is NP-complete to decide whether one of the color classes of a bipartite graph can be covered by a $\mathtt{V}$-free $2$-matching.
\end{abstract}

\section{Introduction} \label{sec:intro}

Throughout the paper, graphs are assumed to be simple. Given an undirected graph $G=(V,E)$ and a subset $F\subseteq E$ of edges, $F(v)$ denotes the set of edges in $F$ incident to a node $v\in V$, and $d_F(v):=|F(v)|$ is the \textbf{degree} of $v$ in $F$. We say that $F$ \textbf{covers} a subset of nodes $X\subseteq V$ if $d_F(v)\geq 1$ for every $v\in X$. Let $b:V\rightarrow\mathbb{Z}_+$ be an upper bound function. A subset $N\subseteq E$ of edges is called a {\bf $b$-matching} if $d_N(v)$ is at most $b(v)$ for every node $v\in V$. For some integer $t\ge 2$, by a {\bf $t$-matching} we mean a $b$-matching where $b(v)=t$ for every $v\in V$. If $t=1$, then a $t$-matching is simply called a \textbf{matching}.

A \textbf{hypergraph} is a pair $H=(V,\cE)$ where $V$ is a finite set
of nodes and $\cE$ is a collection of subsets of $V$. The members of
$\cE$ are called \textbf{hyperedges}, and for a hyperedge $e\in \cE$
let $|e|$ denote its cardinality (as a subset of $V$). In hypergraphs
--unlike in graphs-- we will allow hyperedges of cardinality 1 in this
paper. A \textbf{matching} in a hypergraph is a collection of pairwise
disjoint hyperedges, and the matching is said to be perfect if the
union of the hyperedges in the matching contains every node. The
\textbf{hypergraph matching problem} is to decide whether a given hypergraph has a perfect matching. Given a hypergraph $H=(V,\cE)$,
we can represent it as a bipartite graph $G_H=(U_V, U_\cE; E)$, where
nodes of $U_V$ correspond to nodes in $V$, nodes in $U_\cE$ correspond
to hyperedges in $\cE$, and there is an edge in $G$ between a node
$u_v\in U_V$ (corresponding to $v\in V$) and a node $u_e\in U_\cE$
(corresponding to $e\in \cE$) if and only if $v\in e$ ($G_H$ is also
called the Levi graph of $H$).

Let $G=(S,T;E)$ be a bipartite graph. A path $P=(\{u,v,w\}, \{uv,vw\})$ of length
$2$ with $u,w\in S$ is called an \textbf{$S$-link}, and a
\textbf{$T$-link} can be defined analogously. In \cite{Liang13}, Liang
proposed the following conjecture and showed that, if it is true, the
conjecture implies that $4$-regular graphs are antimagic (where a
simple graph $G=(V,E)$ is said to be \textbf{antimagic} if there exists a bijection $f:E\to \{1,2,\dots, |E|\}$ such that $\sum_{e\in E(v_1)}f(e)\ne \sum_{e\in E(v_2)}f(e)$ for every pair $v_1,v_2\in V$).

\begin{conj} \label{conj:liang}
Assume that $G=(S,T;E)$ is a bipartite graph such that each node in $S$ has degree at most $4$ and each node in $T$ has degree at most $3$. Then $G$ has a matching $M$ and a family $\mathcal{F}$ of node-disjoint $S$-links such that every node $v\in T$ of degree $3$ is covered by an edge in $M\cup(\cup_{P\in\mathcal{F}} P)$.
\end{conj}

Observe that it suffices to verify the conjecture for the special case when each node in $T$ has degree exactly $3$, as we can simply delete nodes of degree less than $3$. Although it was recently proved that regular graphs are antimagic \cite{BBV15}, we prove the conjecture in Section~\ref{sec:ext} as it is interesting in its own. The proof is based on a weakening of the hypergraph matching problem.

While working on the proof of the conjecture, an interesting restricted path factor problem came to our attention. For simplicity, we will call a $T$-link a \textbf{$\mathtt{V}$-path} (the name comes from the shape of these paths when $T$ is placed `above' $S$, see Figure \ref{fig:conj} for an illustration). It is easy to see that a $2$-matching consists of pairwise node-disjoint paths and cycles. We call a $2$-matching  \textbf{$\mathtt{V}$-free} if it does not contain a $\mathtt{V}$-path as a connected component.

\begin{figure}
\begin{center}
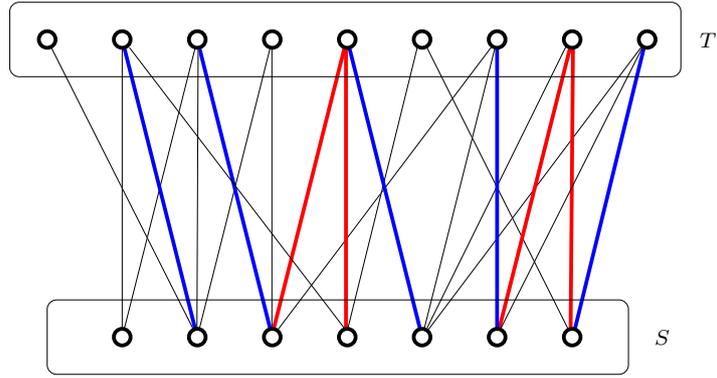
\caption{An illustration for Liang's conjecture. Nodes in $T$ have degree at most 3, and those in $S$ have degree at most 4. The matching is highlighted with blue, the family of $S$-links is highlighted with red.}
\label{fig:conj}
\end{center}
\end{figure}

Consider the problem of finding a matching $M$ and a family $\mathcal{F}$ of node-disjoint $S$-links such that $M\cup(\cup_{P\in\mathcal{F}} P)$ covers $T$. We can assume that $M$ does not contain any edge of $\bigcup\cF$, as such edges can be simply deleted from $M$.
Furthermore, we may assume that each node $v\in T$ has degree at most $2$ in $M\cup(\cup_{P\in\mathcal{F}} P)$. Indeed, if a node $v\in T$ has degree $3$ in $M\cup(\cup_{P\in\mathcal{F}} P)$ then it is covered by both $M$ and $(\cup_{P\in\mathcal{F}} P)$, so the edge in $M$ incident to $v$ can be deleted (see Figure \ref{fig:conj}). It is not difficult to see that $M\cup(\cup_{P\in\mathcal{F}} P)$ is a $\mathtt{V}$-free $2$-matching covering $T$ in this case.

Conversely, given an arbitrary $\mathtt{V}$-free $2$-matching $N$ that covers $T$, edges can be left out from $N$ in such a way that the resulting $\mathtt{V}$-free $2$-matching $N'$ still covers $T$ and consists of paths of length $1$ and $4$, the latter having both end-nodes in $T$. Then $N'$ can be partitioned into a matching and a family of node-disjoint $S$-links.

By the above, the problem of finding a matching $M$ and a family $\mathcal{F}$ of node-disjoint $S$-links whose union covers $T$ is equivalent to finding a $\mathtt{V}$-free $2$-matching $N$ that covers $T$. The proof of Conjecture~\ref{conj:liang} shows that these problems can be solved when  nodes in $S$  have degree at most 4, and those in $T$ have degree at most  $3$. However, in Section~\ref{sec:np} we show that the problem of finding a $\mathtt{V}$-free $2$-matching in a bipartite graph $G=(S,T;E)$ covering $T$ is NP-complete in general.

Let us now recall some well known results from matching theory that will be used below.

\begin{thm}\label{thm:max}
In a bipartite graph there exists a matching that covers every node of maximum degree.
\end{thm}

\begin{thm}[Dulmage and Mendelsohn \cite{Dulmage58}]\label{thm:MD}
Given a bipartite graph $G=(S, T; E)$ and subsets $X\subseteq S$, $Y\subseteq T$, if there exist two matchings $M_X$ and $M_Y$ in $G$ such that
$M_X$ covers $X$ and $M_Y$ covers $Y$ then there exists a matching $M$ in $G$ that covers $X\cup Y$.
\end{thm}

\begin{thm}[Gallai-Edmonds Decomposition Theorem for graphs, see eg. \cite{Lovasz}]\label{thm:GE}
Given a graph $G=(V, E)$, let $D$ be the set of nodes which are not
covered by at least one maximum matching of $G$, $A$ be the set of
neighbours of $D$ and $C:=V-(D\cup A)$. Then \textbf{(a)} the
components of $G[D]$ are factor-critical, \textbf{(b)} $G[C]$ has a
perfect matching, and \textbf{(c)} $G$ has a matching covering $A$.
\end{thm}

\bigskip

The paper is organized as follows. Section~\ref{sec:pre} gives a brief overview of earlier results on restricted path packing problems. In Section~\ref{sec:ext}, we introduce a variant of the hypergraph matching problem and prove a general theorem which in turn implies the conjecture. The paper is closed with a complexity result on $\mathtt{V}$-free $2$-matchings in a bipartite graph $G=(S,T, E)$ covering $T$, see Section~\ref{sec:np}.

\section{Previous work} \label{sec:pre}

For a set $\mathcal{F}$ of connected graphs, a spanning subgraph $M$ of a graph $G$ is called an \textbf{$\mathcal{F}$-factor} of $G$ if every component of $M$ is isomorphic to one of the members of $\mathcal{F}$. The \textbf{path} and \textbf{cycle} having $n$ nodes are denoted by $P_n$ and $C_n$, respectively. The \textbf{length} of $P_n$ is $n-1$, the number of its edges.

The problem of packing $\mathcal{F}$-factors is widely studied. Kaneko presented a Tutte-type characterization of graphs admitting a $\{P_n|n\geq 3\}$-factor \cite{kaneko}. Kano, Katona and Kir\'aly \cite{kano04} gave a simpler proof of Kaneko's theorem and also a min-max formula for the maximum number of nodes that can be covered by a $2$-matching not containing a single edge as a connected component. Such a $2$-matching is often called \textbf{$1$-restricted}. These results were further generalized by Hartvigsen, Hell and Szab\'o \cite{hhsz} by introducing the so-called \textbf{$k$-piece packing} problem, where a $k$-piece is a connected graph with highest degree exactly $k$. In contrast with earlier approaches, their result is algorithmic, and so it provides a polynomial time algorithm for finding a $1$-restricted $2$-matching covering a maximum number of nodes. Later Janata, Loebl and Szab\'o \cite{jlsz} described a Gallai-Edmonds type structure theorem for $k$-piece packings and proved that the node sets coverable by $k$-piece packings have a matroidal structure.

In \cite{hartvigsen07}, Hartvigsen considered the edge-max version of the $1$-restricted $2$-matching problem, that is, when a $1$-restricted $2$-matching containing a maximum number of edges is needed. He gave a min-max theorem characterizing the maximum number of edges in such a subgraph, and he also presented a polynomial algorithm for finding one. The notion of $1$-restricted $2$-matchings was generalized by Li \cite{li} by introducing \textbf{$j$-restricted $k$-matchings} that are $k$-matchings with each connected component having at least $j+1$ edges. She considered the node-weighted version of the problem of finding a $j$-restricted $k$-matching in which the total weight of the nodes covered by the edges is maximal and presented a polynomial algorithm for the problem as well as a min-max theorem in the case of $j < k$. She also proved that the problem of maximizing the number of nodes covered by the edges in a $j$-restricted $k$-matching is NP-hard when $j\geq k \geq 2$.

A graph is called \textbf{cubic} if each node has degree $3$. Cycle-factors and path-factors of cubic graphs are well-studied. The fundamental theorem of Petersen states that each $2$-connected cubic graph has a $\{C_n|n\geq 3\}$-factor \cite{petersen}. From Kaneko's theorem it follows that every connected cubic graph has a $\{P_n|n\geq 3\}$-factor. Kawarabayashi, Matsuda, Oda and Ota proved that every $2$-connected cubic graph has a $\{C_n|n\geq 4\}$-factor, and if the graph has order at least six then it also has a $\{P_n|n\geq 6\}$-factor \cite{kawarabayashi}. For bipartite graphs, these results were improved by Kano, Lee and Suzuki by showing that every connected cubic bipartite graph has a $\{C_n|n\geq 6\}$-factor, and if the graph has order at least eight then it also has a $\{P_n|n\geq 8\}$-factor \cite{kcs}.

Although the $\mathtt{V}$-free $2$-matching problem shows lots of similarities to these problems, it does not seem to fit in the framework of earlier approaches.

\section{Extended matchings} \label{sec:ext}

While working on Conjecture \ref{conj:liang}, we arrived at a
relaxation of the hypergraph matching problem that we call the
\textbf{extended matching problem}.  An \textbf{extended matching} of
a hypergraph $H=(V, \cE)$ is a disjoint collection of hyperedges and pairs of
nodes where a pair $(u,v)$ may be used only if there exists a
hyperedge $e\in\mathcal{E}$ with $u,v\in e$. An extended matching is
\textbf{perfect} if it covers the node-set of $H$. Note that one can
decide in polynomial time if a hypergraph has a perfect extended
matching by the results of \cite{cornuejols1988general} (see also Theorem 4.2.16 in
\cite{jacintphd}). Indeed, given a hypergraph $H=(V,\cE)$, consider its
bipartite representation $G_H=(U_V, U_\cE; E)$. Then a perfect extended
matching in $H$ corresponds to a subgraph in $G_H$ in which nodes of
$U_V$ have degree one, and a node $u_e\in U_\cE$ corresponding to
$e\in \cE$ has degree $|e|$, or any even number not greater than
$|e|$.

However, we have found a simple proof of the following result, a special case of the extended matching problem, which implies Conjecture \ref{conj:liang}, as we show below.

\begin{thm}\label{thm:3unif}
In a 3-uniform hypergraph $H=(V, \cE)$ there exists an extended matching that covers the nodes of maximum degree in $H$.
\end{thm}

Theorem \ref{thm:3unif} is the special case of a more general result (Corollary \ref{cor:ext})
that we introduce below.  Before doing so, we show that Theorem \ref{thm:3unif} implies Conjecture~\ref{conj:liang}.

\begin{proof}[Proof of Conjecture~\ref{conj:liang}]
Recall that it suffices to verify the conjecture for graphs $G=(S,T;E)$ with $d_E(v)=3$ for every $v\in T$. Such a $G$ is the incidence graph (or Levi graph) of a $3$-uniform hypergraph $H=(S,\mathcal{E})$ in which each node has degree at most $4$.

Let $S'\subseteq S$ denote the set of nodes having degree $4$ in $H$. By Theorem \ref{thm:3unif}, $H$ has an extended matching covering $S'$. That is, $S'$ can be covered by pairwise node-disjoint $S$-links and $S$-claws of $G$, where an $S$-claw is a star with $3$ edges having its center node in $T$. We denote the edge-set of these $S$-links and claws by $N$.

Let $T'$ be the set of nodes in $T$ not covered by $N$. As $d_{E-N}(v)\leq 3$ for each $v\in S$, $T'$ can be covered by a matching $M$ disjoint from $N$, by Theorem \ref{thm:max}. By leaving out an edge from each $S$-claw of $N$, we get a matching $M$ and a family of $S$-links whose union together covers $T$.
\end{proof}

Let us now introduce and prove a generalization of Theorem \ref{thm:3unif}.
We call a hypergraph $H=(V,\mathcal{E})$
\textbf{oddly uniform} if every hyperedge has odd cardinality. The
\textbf{quasi-degree} of a node $v\in V$ is defined as
$d^-(v):=\sum[|e|-1:\ v\in e\in\mathcal{E}]$, and the hypergraph is
\textbf{$\Delta$-quasi-regular} (or \textbf{quasi-regular} for short) if $d^-(v)=\Delta$ for each $v\in V$
where $\Delta\in\mathbb{Z}_+$. Note that a uniform regular hypergraph
is quasi-regular.

\begin{thm} \label{thm:psm}
Every oddly uniform quasi-regular hypergraph has a perfect extended matching.
\end{thm}
\begin{proof}
Assume that $H=(V,\mathcal{E})$ is an oddly uniform $\Delta$-quasi-regular hypergraph, and let $G=(V,E)$ denote the graph obtained by replacing each hyperedge $e\in\mathcal{E}$ with a complete graph on node-set $e\subseteq V$. That is, there are as many parallel edges between $u$ and $v$ in $E$ as the number of hyperedges containing both $u$ and $v$. Note that the quasi-regularity of $H$ is equivalent to the regularity of $G$.

If $G$ admits a perfect matching $M$, then $M$ is a perfect extended matching of $H$ and we are done.

Assume that $G$ does not have a perfect matching. Take the Gallai-Edmonds decomposition of $G$ into sets $D$, $A$ and $C$ (see Theorem \ref{thm:GE}). 
Let $D_1$ be the union of those connected components of $G[D]$ that span a hyperedge $e\in\mathcal{E}$ in $H$, and $D_2:=D-D_1$. 

\begin{cl}\label{cl:psm}
Every component $K$ of $G[D_1]$ has a perfect extended matching in $H$.
\end{cl}
\begin{proof}
As $K$ is factor-critical, it has a perfect matching after deleting the nodes of any of its odd cycles (including the case when the cycle consists of a single node). Let $e\in\mathcal{E}$ be a hyperedge spanned by $K$. By the above, $G[K-e]$ has a perfect matching, which together with $e$ form a perfect extended matching of $K$, proving the claim.
\end{proof}

\begin{cl}\label{cl:deg}
For every component $K$ of $G[D_2]$ we have $d_{G}(K)\geq\Delta$.
\end{cl}
\begin{proof}
Let $u\in K$ be an arbitrary node. $K$ does not span a hyperedge in $H$, hence for every hyperedge $e$ containing $u$ we have  $e\cap K\neq\emptyset,e\cap A\neq\emptyset$ and $e\subseteq K\cup A$. By the definition of $G$, there are at least $\sum[|e\cap K|\cdot |e\cap A|:\ u\in e\in\mathcal{E}]\geq\sum[|e|-1:\ u\in e\in\mathcal{E}]=\Delta$ edges between $K$ and $A$, thus concluding the proof of the claim.
\end{proof}

Let $G'=(D',A;F)$ denote the bipartite graph obtained from $G$ by deleting the nodes of $C$ and the edges induced by $A$, and by contracting each component of $G[D]$ to a single node (the set of new nodes is denoted by $D'$). Nodes of $D'$ are partitioned into sets $D'_1$ and $D'_2$ accordingly.
As $d_{G'}(v)\leq\Delta$ for each $v\in A$, Claim~\ref{cl:deg} and Theorem~\ref{thm:max} imply that $G'$ has a matching covering $D'_2$. By Theorem \ref{thm:GE} (c), $G'$ has a matching covering $A$, hence the result of Dulmage and Mendelsohn (Theorem \ref{thm:MD}) implies that $G'$ has a matching $M'$ covering $A$ and $D'_2$ simultaneously. Considering $M'$ as a matching in $G$ and using Theorem \ref{thm:GE} (a) and (b), $M'$ can be extended to a matching $M$ of $G$ that covers every node that is in $C\cup A$ or in a component of $G[D]$ that is incident to an edge in $M'$.  By Claim~\ref{cl:psm}, there is an extended matching covering the nodes of the remaining components of $G[D]$, since they fall in $D_1$. The union of $M$ and this extended matching forms a perfect extended matching of $H$. This completes the proof of the theorem.
\end{proof}

As a consequence, we get the following result.

\begin{cor} \label{cor:ext}
Every oddly uniform hypergraph has an extended matching that covers the set of nodes having maximum quasi-degree.
\end{cor}
\begin{proof}
Let $H=(S,\mathcal{E})$ be an oddly uniform hypergraph and let $\Delta$ denote the maximum quasi-degree in $H$. The \textbf{deficiency} of a node $v\in S$ is $\gamma(v):=\Delta-d^-(v)$. A node $v\in S$ is called \textbf{deficient} if $\gamma(v)>0$. As $H$ is oddly uniform, $\gamma(v)$ is even for every node $v$.

It suffices to show that $H$ can be extended to a $\Delta$-quasi-uniform hypergraph $H'=(V',\mathcal{E}')$ by adding further nodes and hyperedges. Indeed, by Theorem~\ref{thm:psm}, $H'$ admits a perfect extended matching whose restriction to the original hypergraph gives an extended matching covering each node having quasi-degree $\Delta$.

If there is no deficient node in $H$, then we are done. Otherwise consider the hypergraph obtained by taking the disjoint union of three copies of $H$, denoted by $H_1$, $H_2$ and $H_3$, respectively. For each deficient node $v\in S$, add $\gamma(v)$ copies of the hyperedge $\{v_1,v_2,v_3\}$ to the hypergraph, where $v_i$ denotes the copy of $v$ in $H_i$. The hypergraph $H'$ thus obtained is clearly $\Delta$-quasi-regular.
\end{proof}

\section{Complexity result} \label{sec:np}

In what follows we show that deciding the existence of a $\mathtt{V}$-free $2$-matching covering $T$ is NP-complete in general. We will use reduction from the following problem (see \cite[(SP2)]{gj}).

\begin{thm}[$3$-dimensional matching]
Let $H=(X,Y,Z;\mathcal{E})$ be a tripartite $3$-regular $3$-uniform hypergraph, meaning that each node $v\in X\cup Y\cup Z$ is contained in exactly $3$ hyperedges, and each hyperedge $e\in \mathcal{E}$ contains exactly one node from all of $X,Y$ and $Z$. It is NP-complete to decide whether $H$ has a perfect matching, that is, a $1$-regular sub-hypergraph.
\end{thm}

Our proof is inspired by the construction of Li for proving the NP-hardness of maximizing the number of nodes covered by the edges in a $2$-restricted $2$-matching \cite{li}.

\begin{thm} \label{thm:covernp}
Given a bipartite graph $G=(S,T;E)$ with maximum degree $4$, it is NP-complete to decide whether $G$ has a $\mathtt{V}$-free $2$-matching covering $T$.
\end{thm}
\begin{proof}
We prove the theorem by reduction from the $3$-dimensional matching problem. Take a $3$-uniform $3$-regular tripartite hypergraph $H=(X,Y,Z;\mathcal{E})$. For a hyperedge $e\in\mathcal{E}$, we use the following notions: $x_e:=e\cap X$, $y_e:=e\cap Y$ and $z_e:=e\cap Z$.

\begin{figure}[ht]
\centering
\includegraphics[width=0.6\textwidth]{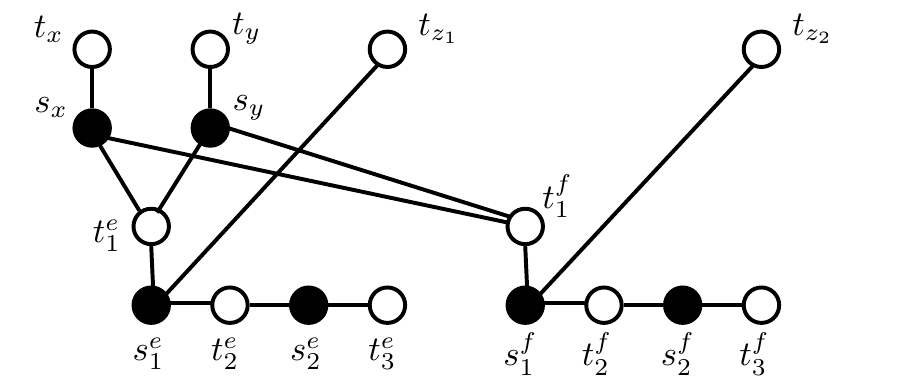}
\caption{Gadgets corresponding to hyperedges $e=\{x,y,z_1\}$ and $f=\{x,y,z_2\}$}
\label{fig:gadget}
\end{figure}

We construct an undirected bipartite graph as follows. For each node $x\in X$ and $y\in Y$, add a pair of nodes $s_x,t_x$ and $s_y,t_y$ to $G$, respectively, with $s_x,s_y\in S$ and $t_x,t_y\in T$. For each node $z\in Z$, add a single node $t_z$ to $T$. Furthermore, for each $x\in X$ and $y\in Y$ add the edges $s_xt_x$ and $s_yt_y$ to $E$.

We assign a path $P_e$ with node set $V(P_e)=\{t^e_1,s^e_1,t^e_2,s^e_2,t^e_3\}$ and edge set $E(P_e)=\{t^e_1s^e_1,s^e_1t^e_2,t^e_2s^e_2,s^e_2t^e_3,t^e_3\}$ of length four to each hyperedge $e\in\mathcal{E}$ and add edges $s_{x_e}t^e_1$, $s_{y_e}t^e_1$ and $t_{z_e}s^e_1$ to $E$ (see Figure~\ref{fig:gadget}). It is easy to check that the graph thus arising is bipartite and has maximum degree $4$ (here we use that every node $v\in X\cup Y\cup Z$ is contained in exactly $3$ hyperedges of $H$).

We claim that $H$ admits a perfect matching if and only if $G$ has a $\mathtt{V}$-free $2$-matching covering $T$, which proves the theorem. Assume first that $H$ has a perfect matching and let $\mathcal{M}\subseteq \mathcal{E}$ be the set of matching hyperedges. Then
\begin{equation*}
M:=\bigcup_{e \in \mathcal{M}}\{s_{x_e}t_{x_e},s_{y_e}t_{y_e},s_{x_e}t^e_1,s_{y_e}t^e_1,t_{z_e}s^e_1,E(P_e)-t^e_1s^e_1\}\cup\bigcup_{e\not\in \mathcal{M}}E(P_e)
\end{equation*}
is a $\mathtt{V}$-free $2$-matching covering $T$ (see Figure~\ref{fig:vfree}).

\begin{figure}[h]
\centering
\includegraphics[width=0.6\textwidth]{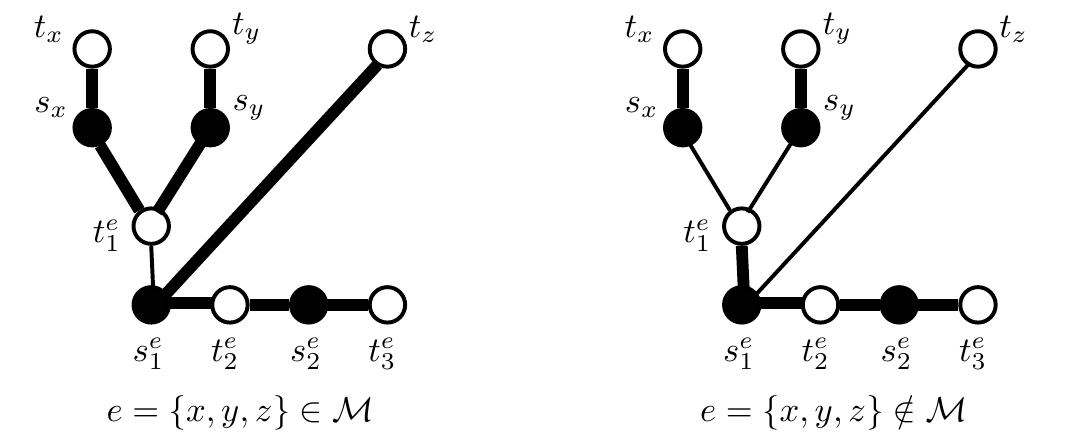}
\caption{Edges included in $M$ depending on whether $e\in\mathcal{M}$ or not}
\label{fig:vfree}
\end{figure}

For the other direction, take a $\mathtt{V}$-free $2$-matching $M$ of $G$ covering $T$. Observe that $s_xt_x,s_yt_y\in M$ for each $x\in X$ and $y\in Y$ as $M$ covers $T$. Moreover, $M$ is $\mathtt{V}$-free hence $t^e_1s^e_1\not\in M$ implies $s_{e_x}t^e_1,y_{e_x}t^e_1\in M$. We may assume that $E(P_e)-t^e_1s^e_1\subseteq M$ for each $e\in\mathcal{E}$. Indeed, $M$ has to cover $t^e_2$ and $t^e_3$, hence the $\mathtt{V}$-freeness of $M$ implies $s^e_1t^e_2,s^e_2t^e_3\in M$. Consequently, $t^e_2s^e_2\in M$ can be assumed.

We claim that $d_M(t_z)=1$ for each $z\in Z$. Indeed, if $t_{z_e}s^e_1\in M$ for some $e\in \mathcal{E}$ then $s_{x_e}t^e_1,s_{y_e}t^e_1\in M$. In other words, if $t_{z_e}s^e_1\in M$ then $e$ `reserves' nodes $s_{x_e},s_{y_e}$ and $t_{z_e}$ for $M$ being a $\mathtt{V}$-free $2$-matching. On the other hand, for each $x\in X$ there is at most one $e\in\mathcal{E}$ such that $s_{x_e}t^e_1\in M$, and the same holds for each $y\in Y$. As the hypergraph is $3$-uniform and $3$-regular, we have $|X|=|Y|=|Z|$. Hence the number of edges of form $t_{z_e}s^e_2$ in $M$ can not exceed the cardinality of these sets. Let
\begin{equation*}
\mathcal{M}:=\{e\in\mathcal{E}:\ t_{z_e}s^e_2\in M\}.
\end{equation*}
By the above, $\mathcal{M}$ is a $1$-regular subhypergraph, thus concluding the proof.
\end{proof}

\section{Acknowledgement}

The first and the second authors were supported by the Hungarian Scientific Research Fund - OTKA, K109240.

\bibliographystyle{abbrv}
\bibliography{v-free}

\end{document}